\documentclass[a4paper,10pt]{amsart}

\usepackage{amsmath,amsfonts,amssymb}
\usepackage[english]{babel}
\usepackage[T1]{fontenc}
\usepackage{lmodern}
\usepackage{graphicx}
\usepackage[all]{xy}

\usepackage[hypertexnames=false,
backref=page,
    pdftex,
    pdfpagemode=UseNone,
    breaklinks=true,
    extension=pdf,
    colorlinks=true,
    linkcolor=blue,
    citecolor=blue,
    urlcolor=blue,
]{hyperref}

\newcommand{\referenza}{}

\newtheorem{prop}{Proposition}[section]
\newtheorem*{prop*}{Proposition \referenza}
\newtheorem{thm}[prop]{Theorem}
\newtheorem*{thm*}{Theorem \referenza}
\newtheorem{cor}[prop]{Corollary}
\newtheorem*{cor*}{Corollary \referenza}

\theoremstyle{definition}
\newtheorem{defi}[prop]{Definition}
\newtheorem*{defi*}{Definition \referenza}
\newtheorem{rmk}[prop]{Remark}
\newtheorem*{rmk*}{Remark \referenza}
\newtheorem{example}[prop]{Example}
\newtheorem*{example*}{Example \referenza}

\newcommand{\N}{\mathbb{N}}
\newcommand{\C}{\mathbb{C}}

\DeclareMathOperator{\im}{i}
\DeclareMathOperator{\imm}{im}

\newcommand{\del}{\partial}
\newcommand{\delbar}{\overline{\del}}

\title{On small deformations of balanced manifolds}

\author{Daniele Angella}
\address[Daniele Angella]{Centro di Ricerca Matematica ``Ennio de Giorgi''\\
Collegio Puteano, Scuola Normale Superiore\\
Piazza dei Cavalieri 3\\
56126 Pisa, Italy
}
\email{daniele.angella@sns.it}
\email{daniele.angella@gmail.com}

\author{Luis Ugarte}
\address[Luis Ugarte]{Departamento de Matem\'aticas\\
Universidad de Zaragoza\\
Edificio de Matem\'aticas\\
c/ Pedro Cerbuna 12, 50009\\
Zaragoza, Spain}
\email{ugarte@unizar.es}

\thanks{During the preparation of the work, the first author has been granted by a research fellowship by Istituto Nazionale di Alta Matematica INdAM and by a Junior Visiting Position at Centro di Ricerca Matematica ``Ennio de Giorgi''; he is also supported by the Project PRIN ``Varietà reali e complesse: geometria, topologia e analisi armonica'', by the Project FIRB ``Geometria Differenziale e Teoria Geometrica delle Funzioni'', by SNS GR14 grant ``Geometry of non-K\"ahler manifolds'', and by GNSAGA of INdAM.
The second author is supported by the projects MINECO (Spain) MTM2011-28326-C02-01 and MTM2014-58616-P,
and by Gobierno de Arag\'on/Fondo Social Europeo, grupo consolidado E15-Geometr\'{\i}a.}

\begin{document}

\begin{abstract}
We introduce a property of compact complex manifolds under which the existence of balanced metric is stable by small deformations of the complex structure. This property, which is weaker than the $\partial\overline\partial$-Lemma, is characterized in terms of the strongly Gauduchon cone and of the first $\partial\overline\partial$-degree measuring the difference of Aeppli and Bott-Chern cohomologies with respect to the Betti number $b_1$.
\end{abstract}

\maketitle

\section*{Introduction}

\noindent In this note we are aimed at the problem of constructing special metrics on complex non-K\"ahler manifolds. In particular, we are interested in {\em balanced metrics} in the sense of M.~L. Michelsohn \cite{michelsohn}, that is, Hermitian metrics whose fundamental form is co-closed. More precisely, basing on the work by J. Fu and S.-T. Yau \cite{fu-yau}, we introduce a condition ensuring the existence of such metrics on small deformations of the complex structure.

\medskip

It is well known that the existence of balanced metrics on a compact complex manifold is not stable under small deformations of the complex structure~\cite{AB1}. More precisely, \cite[Proposition 4.1]{AB1} provides a counter-example on the Iwasawa manifold endowed with the 
holomorphically parallelizable complex structure.
In fact, in order to prove a balanced analogue of the fundamental stability result by K. Kodaira and D.~C. Spencer \cite[Theorem 15]{kodaira-spencer}, one needs a further assumption on the variation of Bott-Chern cohomology.
We then investigate cohomological conditions on the central fibre yielding existence of balanced metrics for small deformations.

As a first result in this direction, C.-C. Wu proves in \cite[Theorem 5.13]{wu} that small deformations of
compact complex manifolds satisfying the $\del\delbar$-Lemma and admitting balanced metrics still admit balanced metrics.

The assumption on the validity of $\del\delbar$-Lemma is in fact stronger than necessary. It suffices, for example, that the dimension of the $(n-1,n-1)$-th Bott-Chern cohomology group is constant along the deformation, where $2n$ denotes the real dimension of the manifold, see Proposition \ref{prop:stability-under-const-dim-bc}.
This condition, too, is sufficient but not necessary, see Example \ref{example1}.
The above condition allows to show that any small deformation of the Iwasawa manifold endowed with an Abelian complex structure admits balanced metrics, see Proposition \ref{prop:iwasawa-ab-def}, a result that is in deep contrast with the behaviour of its holomorphically parallelizable structure. 

In order to generalize Wu's result, J. Fu and S.-T. Yau introduced in \cite[Definition 5]{fu-yau} the following finer notion.
A compact complex manifold $X$ of complex dimension $n$ is said to satisfy the {\em $(n-1,n)$-th weak $\del\delbar$-Lemma} if, for each real form $\alpha$ of type $(n-1,n-1)$ on $X$ such that $\delbar\alpha$ is $\del$-exact, then there exists a $(n-2,n-1)$-form $\beta$ such that $\delbar\alpha=\im\,\del\delbar\beta$. They proved the following result.

\begin{thm}[{\cite[Theorem 6]{fu-yau}}]\label{fu-yau}
Let $X$ be a compact complex manifold of complex dimension $n$ with a balanced metric, and let $X_{t}$ be a holomorphic deformation of $X=X_0$.
If $X_{t}$ satisfies the $(n-1,n)$-th weak $\del\delbar$-Lemma for any $t\not=0$,
then there exists a balanced metric on $X_{t}$ for $t$ sufficiently close to $0$.
\end{thm}

Notice that, while satisfying the $\del\delbar$-Lemma is a stable property under small deformations of the complex structure,
(see \cite[Proposition 9.21]{voisin}, or \cite[Theorem 5.12]{wu}, or \cite[Corollary 2.7]{angella-tomassini-3},) satisfying the $(n-1,n)$-th weak $\del\delbar$-Lemma is not open under deformations, (see \cite[Example 3.7]{ugarte-villacampa}).
Here we propose a cohomological notion, related to the above weak $\del\delbar$-Lemma, in order to get stability under small deformations.

\medskip

Let $X$ be a compact complex manifold of complex dimension~$n$.
We recall that the Bott-Chern and the Aeppli cohomologies \cite{Aeppli, bott-chern} of $X$ are defined, respectively, by
$$ H^{\bullet,\bullet}_{BC}(X):=\frac{\ker\del\cap\ker\delbar}{\imm\del\delbar} \qquad \text{ and }
\qquad H^{\bullet,\bullet}_{A}(X):=\frac{\ker\del\delbar}{\imm\del+\imm\delbar} \;. $$
Consider the natural map
$$\iota_{BC,\,A}^{n-1,\,n} \colon H^{n-1,\,n}_{BC}(X) \to H^{n-1,\,n}_{A}(X) $$
induced by the identity.

We introduce the following notion.

\renewcommand{\referenza}{\ref{def:(n-1,n)-strong-deldelbar-lemma}}
\begin{defi*}
A compact complex manifold $X$ of complex dimension~$n$ is said to satisfy 
the {\em $(n-1,n)$-th strong $\del\delbar$-Lemma} if the natural map $\iota_{BC,\,A}^{n-1,\,n}$ is injective.
\end{defi*}

It is clear that compact complex manifolds satisfying the $\del\delbar$-Lemma also satisfy the
$(n-1,n)$-th strong $\del\delbar$-Lemma. In Proposition~\ref{relaciones},
we prove that the latter property is weaker than the $\del\delbar$-Lemma.
It is also clear that the $(n-1,n)$-th strong $\del\delbar$-Lemma implies the $(n-1,n)$-th weak $\del\delbar$-Lemma.
Proposition~\ref{relaciones} shows that the converse does not hold.

\medskip

The main result in this note is a characterization of the $(n-1,n)$-th strong $\del\delbar$-Lemma.
More precisely, in Theorem~\ref{main-equiv}, we prove that a compact complex manifold $X$ satisfies the $(n-1,n)$-th strong $\del\delbar$-Lemma if and only if
the strongly Gauduchon cone of $X$ coincides with its Gauduchon cone and the first $\partial\overline\partial$-degree $\Delta^1(X)$ vanishes.
Compact complex manifolds satisfying the first condition are called {\em sGG manifolds} and they are studied in \cite{popovici-ugarte}. On the other hand, the complex invariants $\Delta^k(X)$, for $1 \leq k \leq n$, to which we refer here as the {\em $k$-th $\partial\overline\partial$-degrees} of $X$, are introduced in \cite{angella-tomassini-3}, where it is proved that they all vanish if and only if the compact complex manifold $X$ satisfies the $\partial\overline\partial$-Lemma.

Such a characterization turns out to be open under deformations of the complex structure, Proposition \ref{openness}. Hence, we get that the $(n-1,n)$-th strong $\del\delbar$-Lemma provides a condition assuring stability of the existence of balanced metrics. This is the final aim of this note.

\renewcommand{\referenza}{\ref{conseq}}
\begin{thm*}
Let $X$ be a compact complex manifold of complex dimension $n$ with a locally conformally balanced metric, and let $X_{t}$ be a holomorphic deformation of $X=X_0$.
If $X$ satisfies the $(n-1,n)$-th strong $\del\delbar$-Lemma, then $X_{t}$ admits a balanced metric for any $t$ sufficiently close to $0$.
\end{thm*}

The proof follows by noticing that, by \cite[Theorem 2.5]{angella-ugarte-1}, the $(n-1,n)$-th strong $\del\delbar$-Lemma property ensures that locally conformally balanced structures are in fact globally conformal to a balanced structure, hence yielding the existence of a balanced metric; and finally by applying Fu and Yau's result, Theorem~\ref{fu-yau}.

\bigskip

\noindent{\sl Acknowledgments.}
The first author would like to thank the Departamento de Matem\'aticas of the Universidad de Zaragoza for the warm hospitality.

\section{Preliminaries on sGG manifolds}

\noindent Let $X$ be a compact complex manifold of complex dimension $n$.
We recall that the \emph{Gauduchon cone} ${\mathcal{C}_G}(X)$ of $X$ is defined in \cite{popovici} as the open convex cone
$$
{\mathcal{C}_G}(X) \;\subset\; H^{n-1,\,n-1}_A(X)
$$
consisting of the (real) Aeppli cohomology classes $[\omega^{n-1}]_A$ which are represented by $(n-1)$-powers of Gauduchon metrics~$\omega$ on $X$ (that is, Hermitian metrics $\omega$ satisfying $\del\delbar\omega^{n-1}=0$).

A Gauduchon metric $\omega$ is called {\em strongly Gauduchon}, \cite[Definition 4.1]{popovici-inventiones},
if $\del\omega^{n-1}$ is $\delbar$-exact.
Consider the map $T$, induced by $\del$ in cohomology, given by
\begin{equation}\label{T-def}
T \colon H^{n-1,\,n-1}_{A}(X) \to H^{n,\,n-1}_{\delbar}(X) \;, \qquad T([\Omega]_{A})\;:=\;[\del\Omega]_{\delbar} \;,
\end{equation}
for any $[\Omega]_A \in H^{n-1,\,n-1}_{A}(X)$.
The \emph{strongly Gauduchon cone} (sG cone, for short) ${\mathcal{C}_{sG}}(X)$
is defined in~\cite{popovici} as
the intersection of the Gauduchon cone with the kernel of the linear map~$T$, i.e.,
$$
{\mathcal{C}_{sG}}(X) \;:=\; {\mathcal{C}_G}(X)\cap\ker T \;\subseteq\; {\mathcal{C}_G}(X) \;\subset\; H^{n-1,\,n-1}_A(X)\;.
$$
Notice that the sG property is cohomological. In fact, either all the Gauduchon metrics $\omega$ for which $\omega^{n-1}$ belongs to a given Aeppli-Gauduchon class $[\omega^{n-1}]_{A}\in {\mathcal{C}_G}(X)$ are sG, or none of them is.

\medskip

The following class is introduced and studied in \cite{popovici-ugarte}:
a compact complex manifold $X$ is said to be an \emph{sGG manifold} if the sG cone of $X$ coincides with the Gauduchon cone,
i.e., $${\mathcal{C}_{sG}}(X) \;=\; {\mathcal{C}_G}(X)\;.$$

We will need the following conditions equivalent to the sGG property.

\begin{thm}[{\cite[Observation 5.3]{popovici}, \cite[Lemma 1.2, Theorem 1.3, Theorem 1.5]{popovici-ugarte}}]\label{car-sGG}
Let $X$ be a compact complex manifold of complex dimension $n$. The following statements are equivalent:
\begin{enumerate}
\item[{\rm ({\it i})}] $X$ is an sGG manifold;
\item[{\rm ({\it ii})}] every Gauduchon metric $\omega$ on $X$ is strongly Gauduchon;
\item[{\rm ({\it iii})}] the natural map $\iota^{n,\,n-1}_{\delbar,\,A}\colon H^{n,n-1}_{\delbar}(X)\to H^{n,n-1}_{A}(X)$
induced by the identity is injective;
\item[{\rm ({\it iv})}] the map $T$ given by \eqref{T-def} vanishes identically;
\item[{\rm ({\it v})}] the following special case of the $\del\delbar$-Lemma holds: for every $d$-closed form
$\Omega$ of type $(n,n-1)$ on $X$, if $\Omega$ is $\del$-exact, then $\Omega$ is also $\delbar$-exact;
\item[{\rm ({\it vi})}] there holds $h^{0,1}_{\mathrm{BC}}(X) = h^{0,1}_{\delbar}(X)$;
\item[{\rm ({\it vii})}] there holds $b_1 = 2\,h^{0,1}_{\delbar}(X)$.
\end{enumerate}
\end{thm}

(Here $b_k$ denotes de $k$-th Betti number of the manifold,
and $h^{p,q}_{\delbar}(X)$, $h^{p,q}_{BC}(X)$ and $h^{p,q}_{A}(X)$ denote, respectively,
the dimensions of the Dolbeault, Bott-Chern and Aeppli cohomology groups of $X$.)

\medskip

Finally, we recall the definition of the {\em $\del\delbar$-degrees} $\Delta^k(X)$ of $X$.
In \cite[Theorem A]{angella-tomassini-3}, it is proven that, for any $k \in \{1,\ldots,2n\}$,
$$ \Delta^k(X) \;:=\; \sum_{p+q=k} \left( h^{p,q}_{BC}(X) + h^{p,q}_{A}(X) \right) - 2\, b_k \;\in\; \N $$
are non-negative integers.
Furthermore, it is proven in \cite[Theorem B]{angella-tomassini-3} that $X$ satisfies the $\del\delbar$-Lemma
if and only if $\Delta^k(X)=0$ for any $k$.
Notice that since $h^{p,q}_{BC}(X)=h^{q,p}_{BC}(X)=h^{n-q,n-p}_{A}(X)=h^{n-p,n-q}_{A}(X)$ by \cite[\S2.c]{schweitzer},
we have $\Delta^{2n-k}(X)=\Delta^k(X)$ for any $k$.

In this note, the invariant
\begin{equation}\label{delta1}
\frac12\, \Delta^1(X) \;=\; h^{0,1}_{BC}(X)+h^{0,1}_{A}(X)-b_1 \;\in\; \N \;
\end{equation}
will play a central role.

\section{\texorpdfstring{$(n-1,n)$}{(n-1,n)}-th strong \texorpdfstring{$\del\delbar$}{partial-overline-partial}-Lemma}

\noindent We study the following class of compact complex manifolds.

\begin{defi}\label{def:(n-1,n)-strong-deldelbar-lemma}
Let $X$ be a compact complex manifold of complex dimension~$n$. We say that $X$ satisfies the {\em $(n-1,n)$-th strong $\del\delbar$-Lemma} if the natural map $\iota_{BC,\,A}^{n-1,\,n}$ is injective.
\end{defi}

Equivalently, for each $\del$-closed form $\Gamma$ of type $(n-1,n)$ on $X$, if $\Gamma=\del \eta + \delbar \nu$, then there exists a $(n-2,n-1)$-form $\gamma$ such that $\Gamma=\del\delbar \gamma$.

\medskip

Clearly, any compact complex manifold satisfying the $\del\delbar$-Lemma also satisfies the $(n-1,n)$-th strong $\del\delbar$-Lemma, and it is also clear that the latter condition implies the $(n-1,n)$-th weak $\del\delbar$-Lemma.
In the following result we show that the converses to these implications do not hold.

\begin{prop}\label{relaciones}
There exist:
\begin{enumerate}
 \item compact complex manifolds $X$ of complex dimension $n$ satisfying the $(n-1,n)$-th strong $\del\delbar$-Lemma
that do not satisfy the $\del\delbar$-Lemma;
 \item compact complex manifolds $X$ of complex dimension $n$ satisfying the $(n-1,n)$-th weak $\del\delbar$-Lemma
but not satisfying the $(n-1,n)$-th strong $\del\delbar$-Lemma.
\end{enumerate}
\end{prop}

\begin{proof}
As an example in case (i), let $X$ be the compact complex manifold of complex dimension $3$ 
given by the completely-solvable Nakamura manifold
with the lattice in case {\itshape (ii)} in \cite[Example 2.17]{angella-kasuya-1} 
(see Example~\ref{example1} below for details).
Then, $b_1 =2$ and $h^{0,1}_{\delbar}(X) = h^{0,1}_{BC}(X) = h^{0,1}_{A}(X)=1$,
so Corollary~\ref{numerical-char} below implies that $X$ satisfies the $(2,3)$-th strong $\del\delbar$-Lemma.
However, $\Delta^2(X)=4 \not=0$ and $X$ does not satisfy the $\del\delbar$-Lemma \cite{angella-tomassini-3}.

For the proof of (ii), we first observe that, by the commutative diagram of natural maps
\begin{equation}\label{diagg}
\xymatrix{
  H^{n-1,n}_{BC}(X) \ar[r]^{\iota^{n-1,\,n}_{BC,\,\delbar}} \ar@/_15pt/[rr]_{\iota^{n-1,\,n}_{BC,\,A}} & H^{n-1,n}_{\delbar}(X) \ar[r]^{\iota^{n-1,\,n}_{\delbar,\,A}} & H^{n-1,n}_{A}(X) \, ,
 }
\end{equation}
if a compact complex manifold $X$ satisfies the $(n-1,n)$-th strong $\del\delbar$-Lemma
then the map $\iota^{n-1,\,n}_{BC,\,\delbar}$ is injective.

Let $X$ be a nilmanifold endowed with an Abelian complex structure.
By \cite[Corollary 3.5]{ugarte-villacampa}, $X$ always satisfies the $(n-1,n)$-th weak $\del\delbar$-Lemma;
on the other hand, in \cite[Proposition 2.9]{angella-ugarte-1} it is proved that the map $\iota^{n-1,\,n}_{BC,\,\delbar}$ is never injective.
Hence, $X$ does not satisfy the $(n-1,n)$-th strong $\del\delbar$-Lemma.
\end{proof}

\section{A characterization of the \texorpdfstring{$(n-1,n)$}{(n-1,n)}-th strong \texorpdfstring{$\del\delbar$}{partial-overline-partial}-Lemma}

\noindent In this section, we provide a characterization of the $(n-1,n)$-th strong $\del\delbar$-Lemma in terms of the sGG property and the vanishing of the complex invariant~$\Delta^1$.
This will allow us to ensure the openness of $(n-1,n)$-th strong $\del\delbar$-Lemma under small deformations of the complex structure in the next section.

\begin{thm}\label{main-equiv}
Let $X$ be a compact complex manifold of dimension $n$. Then, $X$ satisfies the $(n-1,n)$-th strong $\del\delbar$-Lemma if and only if $X$ is sGG and $\Delta^1(X)=0$.
\end{thm}

\begin{proof}
We divide the proof into the following steps.

\smallskip
\noindent \textbf{Step 1.} {\itshape If $X$ satisfies the
 $(n-1,n)$-th strong $\del\delbar$-Lemma,
 then $X$ is sGG.}\\
\noindent
Let $\omega$ be any Gauduchon metric on $X$. Then, $\Omega=\delbar \omega^{n-1}$ yields a class in $H^{n-1,n}_{BC}(X)$ such that $\iota_{BC,\,A}^{n-1,\,n}([\Omega]_{BC})=0$ in $H^{n-1,n}_{A}(X)$. The injectivity of $\iota_{BC,\,A}^{n-1,\,n}$
implies the existence of a form $\alpha$ such that $\Omega=\del\delbar \alpha$.
Taking $\gamma:=\delbar\alpha$ we have that $\delbar \omega^{n-1}=\Omega=\del \gamma$. Therefore, $\del\omega^{n-1}=\overline{\delbar\omega^{n-1}}=\delbar\overline\gamma$, so the metric
$\omega$ is strongly Gauduchon. Now, by Theorem~\ref{car-sGG}, the manifold $X$ is sGG.

\smallskip
\noindent \textbf{Step 2.} {\itshape If $X$ satisfies the
 $(n-1,n)$-th strong $\del\delbar$-Lemma,
 then $\Delta^1(X)=0$.}\\
\noindent The injectivity of the map $\iota_{BC,\,A}^{n-1,\,n}$ implies
$h^{1,0}_{A}(X) \;=\; h^{n-1,n}_{BC}(X) \;\leq\; h^{n-1,n}_{A}(X) \;=\; h^{1,0}_{BC}(X)$,
where we have used the duality between Aeppli cohomology and Bott-Chern cohomology \cite{schweitzer}.
Then by \eqref{delta1} we get
 \begin{eqnarray*}
 0 \;\leq\; \frac12\,\Delta^1(X) &=& h^{1,0}_{BC}(X) + h^{1,0}_{A}(X) - b_1 \\[5pt]
 &\leq& 2\, h^{1,0}_{BC}(X) - b_1 \\[5pt]
 &=& 2\, h^{0,1}_{BC}(X) - b_1 \\[5pt]
 &=& 2\, h^{0,1}_{\delbar}(X) - b_1 \;=\; 0 \;,
 \end{eqnarray*}
where we have used {\it Step 1}, and ({\it vi}) and~({\it vii}) in Theorem~\ref{car-sGG}.
Hence, $\Delta^1(X)=0$.

\smallskip
\noindent \textbf{Step 3.} {\itshape If $X$ is sGG and $\Delta^1(X)=0$, then $X$ satisfies the
 $(n-1,n)$-th strong $\del\delbar$-Lemma.
 }\\
 \noindent
Consider the commutative diagram of natural maps
 $$ \xymatrix{
  H^{n-1,n}_{BC}(X) \ar[r]^{\iota^{n-1,\,n}_{BC,\,\del}} \ar@/_15pt/[rr]_{\iota^{n-1,\,n}_{BC,\,A}} & H^{n-1,n}_{\del}(X) \ar[r]^{\iota^{n-1,\,n}_{\del,\,A}} & H^{n-1,n}_{A}(X) \;.
 } $$

By the assumption that $X$ is sGG we have that the natural map $\iota^{n,\,n-1}_{\delbar,\,A}$ is injective by Theorem~\ref{car-sGG}~({\it iii}), whence the natural map $\iota^{n-1,\,n}_{\del,\,A}$ is injective.
We also have that
$$ h^{n-1,n}_{A}(X) \;=\; h^{1,0}_{BC}(X) \;=\; h^{1,0}_{\del}(X) \;=\; h^{n-1,n}_{\del}(X) \;, $$
where we have used the duality between Aeppli cohomology and Bott-Chern cohomology \cite{schweitzer},
Theorem~\ref{car-sGG}~({\it vi}), and the Serre duality, respectively.
It follows that $\iota^{n-1,\,n}_{\del,\,A}$ is in fact an isomorphism.

By the hypothesis $\Delta^1(X)=0$ we have
 \begin{eqnarray*}
  h^{n-1,n}_{BC}(X) &=& b_1(X) - h^{1,0}_{BC}(X) \,=\, 2\, h^{0,1}_{\delbar}(X) - h^{0,1}_{BC}(X) \\[5pt]
  &=& h^{0,1}_{BC}(X) \,=\,  h^{n-1,n}_{A}(X) ,
 \end{eqnarray*}
where we have used \eqref{delta1}, Theorem~\ref{car-sGG}~({\it vii}), Theorem~\ref{car-sGG}~({\it vi}),
and the duality between Aeppli cohomology and Bott-Chern cohomology \cite{schweitzer}, respectively.

Notice that the natural map $\iota_{BC,\,\del}^{n-1,\,n}$ is always surjective for bidegree reasons.
Since $\iota^{n-1,\,n}_{\del,\,A}$ is an isomorphism, then $\iota^{n-1,\,n}_{BC,\,A}$
is surjective and, by the above equality $h^{n-1,n}_{BC}(X)=h^{n-1,n}_{A}(X)$, it is in fact an isomorphism.
Therefore, $X$ satisfies the $(n-1,n)$-th strong $\del\delbar$-Lemma.
\end{proof}

\begin{rmk}\label{remark1}
The properties sGG and $\Delta^1(X)=0$ of compact complex manifolds $X$ are unrelated.
For instance, the Iwasawa manifold is sGG but $\Delta^1= 2\not=0$.
On the other hand, the product of the $5$-dimensional generalized Heisenberg nilmanifold  by $S^1$ has invariant complex structures satisfying $\Delta^1=0$, but they are never sGG manifolds.
(For general results on complex nilmanifolds satisfying the sGG property see \cite{popovici-ugarte}, and for a general study of the invariants $\Delta^k$ see \cite{AFR,LUV}.)
\end{rmk}

\medskip

By Theorem \ref{main-equiv} and by noting that the natural map $\iota^{0,\,1}_{\delbar,\,A}\colon H^{0,1}_{\delbar}(X)\to H^{0,1}_{A}(X)$ induced by the identity is always injective, we get the following numerical characterization of  the $(n-1,n)$-th strong $\del\delbar$-Lemma in terms of Betti and Aeppli numbers.

\begin{cor}\label{numerical-char}
On any compact complex manifold $X$ of complex dimension $n$,
we have
$$ h^{0,1}_{BC}(X) \;\leq\; h^{0,1}_{\delbar}(X) \;\leq\; h^{0,1}_{A}(X) \quad \text{ and } \quad b_1 \;\leq\; 2\,h^{0,1}_{\delbar}(X) \;.$$
Moreover, $X$ satisfies the $(n-1,n)$-th strong $\del\delbar$-Lemma if and only if
$h^{0,1}_{BC}(X) = h^{0,1}_{A}(X)$, if and only if
$$
b_1 = 2\,h^{0,1}_{A}(X).
$$
\end{cor}

\section{Stability of the balanced condition}

\noindent In this section, we prove two results concerning stability of the balanced condition. In particular, they possibly allow to construct new examples of balanced manifolds.

\subsection{Balanced metrics and variation of Bott-Chern cohomology}

The first result is mainly intended to notice that the property of stability of balanced metric is, in some sense, closely related to the variation of Bott-Chern cohomology. The following proposition is a straightforward extension of the stability result \cite[Theorem 15]{kodaira-spencer} by K. Kodaira and D.~C. Spencer for K\"ahler metrics. The argument is substantially the same as in \cite[Theorem 5.13]{wu}, where the statement is proven with the stronger hypothesis of the $\partial\overline\partial$-Lemma. See also \cite[Theorem 8.11]{cavalcanti-skt}, where a similar argument is used to study stability for SKT metrics.

\begin{prop}\label{prop:stability-under-const-dim-bc}
 Let $X$ be a compact complex manifold, and let $\{X_t\}_{t\in(-\varepsilon,\varepsilon)}$ be a differentiable family of deformations of $X_0=X$, where $\epsilon>0$. If $X$ admits a balanced metric and the upper-semi-continuous function $t\mapsto \dim_\C H^{n-1,n-1}_{BC}(X_t)$ is constant, then $X_t$ admits a balanced metric for any $t$ close enough to $0$.
\end{prop}

\begin{proof}
 Take a family $\{\omega_t\}_t$ of Hermitian metrics on $X_t$.
 For any $t\in(-\varepsilon,\varepsilon)$, consider the Bott-Chern Laplacian $\Delta^{BC}_t$ associated to $\omega_t$ and the corresponding Green operator $G_t$ \cite{schweitzer}. Denote by $\pi_{\wedge^{n-1,n-1}X_t}\colon \wedge^{\bullet}X\otimes\C\to \wedge^{n-1,n-1}X_t$ the projection onto the space of $(n-1,n-1)$-forms on $X_t$, and by $H_t\colon\wedge^{\bullet}X\otimes\C \to \ker\Delta^{BC}_t$ the projection onto the space of harmonic forms with respect to $\Delta^{BC}_t$ (and with respect to the $L^2$-pairing induced by $\omega_t$).

 For any $t\in(-\varepsilon,\varepsilon)$, consider the operator
 $$ \Pi_t \;:=\; \left( H_t + \del_t\delbar_t\left(\del_t\delbar_t\right)^{*_t}G_t \right) \circ \pi_{\wedge^{n-1,n-1}X_t} \colon \wedge^\bullet X\otimes\C \to \ker\del_t\cap\ker\delbar_t \;, $$
 where $*_t$ is the Hodge-$*$-operator with respect to $\omega_t$.
 It gives the projection onto the space of $\del_t$-closed $\delbar_t$-closed $(n-1,n-1)$-forms on $X_t$.
 By elliptic theory, if the function $t\mapsto \dim_\C H^{n-1,n-1}_{BC}(X_t)$ is constant, then the family $\{\Pi_t\}_{t}$ is smooth in $t$, see \cite[Theorem 7]{kodaira-spencer}.

 Now, let $\eta_0$ be a balanced metric on $X_0$. For $t\in(-\varepsilon,\varepsilon)$, set
 $$ \Omega_t \;:=\; \Pi_t \eta_0^{n-1} \;. $$
 The family $\{\Omega_t\}_t$ is smooth in $t$. In particular, since $\Omega_0=\Pi_0\eta_0^{n-1}=\eta_0^{n-1}$ is a positive form, then, for $t$ close enough to $0$, the form $\Omega_t$ is positive, too. By the Michelsohn trick \cite[pages 279--280]{michelsohn}, there exists a Hermitian metric $\eta_t:=\sqrt[n-1]{\Omega_t}$, which is in fact a balanced metric on $X_t$.
\end{proof}

\begin{rmk}
The condition on stability of the dimension of Bott-Chern cohomology in Proposition \ref{prop:stability-under-const-dim-bc} is sufficient but not necessary, see Example \ref{example1} below.
\end{rmk}

The following corollary should be compared with Theorem \ref{conseq}, where the hypothesis uses instead the vanishing of the first $\partial\overline\partial$-degree $\Delta^1(X)$.

\begin{cor}
 Let $X$ be a compact complex manifold admitting balanced metrics and such that $\Delta^2(X)=0$. Then any small deformation admits balanced metrics.
\end{cor}

\begin{proof}
Let $\left\{X_t\right\}_{|t|<\varepsilon}$ be a differentiable family of small deformations of $X=X_0$, where $\varepsilon>0$. 
The statement follows by noting that
\begin{eqnarray*}
  \Delta^2(X_t) \;=\; \Delta^{2n-2}(X_t) &=& \left( h^{n,n-2}_{BC}(X_t) + h^{n-1,n-1}_{BC}(X_t) + h^{n-2,n}_{BC}(X_t) \right) \\[5pt]
  && + \left( h^{n,n-2}_{A}(X_t) + h^{n-1,n-1}_{A}(X_t) + h^{n-2,n}_{A}(X_t) \right) \\[5pt]
  && - 2\, b_{2n-2}(X_t) \;\geq\; 0 \;,
\end{eqnarray*}
where the functions $t\mapsto h^{p,q}_{BC}(X_t)$ and $t\mapsto h^{p,q}_{A}(X_t)$ are upper-semi-continuous for any $(p,q)$, 
and the function $t\mapsto b_{2n-2}(X_t)$ is locally constant. 
In particular, it follows that $t\mapsto \dim_\C H^{n-1,n-1}_{BC}(X_t)$ is constant.
\end{proof}

In fact, in order to show that $h^{n-1,n-1}_{BC}(X_t)$ is constant near $X_0$, it suffices to have that $\dim_\C H^{n-1,n-1}_{BC}(X)$ has its minimum at $X_0$, varying $X$ among deformations of $X_0$. In particular, this applies to the Iwasawa manifold endowed with any Abelian complex structure, showing a different behaviour with respect to the holomorphically parallelizable complex structure on it in terms of stability of existence of balanced metrics, see \cite[Proposition 4.1]{AB1}.

\begin{prop}\label{prop:iwasawa-ab-def}
Let $X$ be the Iwasawa manifold endowed with an Abelian complex structure. 
Let $\{X_t\}_{t\in(-\varepsilon,\varepsilon)}$ be a differentiable family of deformations of $X_0=X$, where $\epsilon>0$. 
Then $X_t$ admits a balanced metric for any $t$ close enough to $0$.
\end{prop}

\begin{proof}
Note that $X$ admits a balanced metric by \cite[Corollary 2.9]{ugarte-villacampa}.
By \cite[Remark 4]{console-fino} and \cite[Theorem 2.6]{rollenske}, the complex structure of any sufficiently small deformation of $X$ is invariant.
From \cite[Table 2]{AFR} and \cite[Appendix 6]{LUV}, the dimension of the Bott-Chern cohomology group of bi-degree $(2,2)$, (which varies upper-semi-continuously along differentiable families,) reaches its minimum value $6$ among invariant complex structures 
on the manifold $X_0$. Hence the statement follows from Proposition \ref{prop:stability-under-const-dim-bc}.
\end{proof}

\begin{example}\label{Iwasawa-with-abelian-J}
We provide here a concrete example of deformations valid for any Abelian complex structure $J$ 
on the Iwasawa manifold. 
The invariant Abelian complex structures on the Iwasawa manifold 
are classified in \cite{ABD} up to isomorphism, and by \cite[Section 3]{COUV} 
there exists a $(1,0)$-basis $\{\eta^j\}_{j\in\{1,2,3\}}$ satisfying
$$
d\eta^1 \;=\; d\eta^2 \;=\; 0 \;,
\quad
d\eta^3 \;=\; \eta^{1\bar{1}} + \eta^{1\bar2} + D\,\eta^{2\bar{2}} \;,
$$
for some $D\in [0,1/4)$.
Since the forms $\eta^{\bar{j}}$, for $j\in\{1,2,3\}$, are $\bar{\partial}$-closed,
the Dolbeault cohomology group of bi-degree $(0,1)$ is given by
$H^{0,1}_{\bar{\partial}} = \C\left\langle \left[\bar\eta^{1}\right], \; \left[\bar\eta^{2}\right], \; \left[\bar\eta^{3}\right] \right\rangle$.
We define the following deformation of $J$: for $t \in \Delta=\left\{ t\in\mathbb{C} \;\middle\vert\; |t|^2<1 \right\}$, we consider the complex structure $J_t$ on the Iwasawa manifold given by the following $(1,0)$-forms:
$$
\eta^1_t \;:=\; \eta^1 \;, \quad
\eta^2_t \;:=\; \eta^2 + t\, \eta^{\bar{2}} \;, \quad
\eta^3_t \;:=\; \eta^3 \;.
$$
It is straightforward to see that the differentials satisfy
$$
d\eta^1_t \;=\; d\eta^2_t \;=\; 0 \;,
\quad
d\eta^3_t \;=\; -\frac{\bar{t}}{1-|t|^2}\,\eta^{12}_t + \eta^{1\bar{1}}_t + \frac{1}{1-|t|^2}\,\eta^{1\bar{2}}_t + \frac{D}{1-|t|^2}\,\eta^{2\bar{2}}_t \;.
$$
Therefore, $J_t$ defines an invariant complex structure for any $t \in \Delta$, which is non-Abelian for any $t \not= 0$.

By Proposition \ref{prop:iwasawa-ab-def}, the complex structures $J_t$ admit a balanced metric for any $t$ close enough to $0$.
In fact, for any $t\in\Delta$, the real $2$-form
$$
\omega_t \;:=\; \frac{\im}{2}\,\eta^{1\bar{1}}_t + \frac{\im}{4}\, \frac{1-2D}{1-|t|^2}\,\eta^{2\bar{2}}_t
+ \frac{\im}{2}\,\eta^{3\bar{3}}_t
+ \frac{\im}{4}\,\eta^{1\bar{2}}_t + \frac{\im}{4}\,\eta^{2\bar{1}}_t
$$
is positive, because $0\leq 4D<1$, and it satisfies that $d \omega_t^2=0$. 
So it defines a balanced $J_t$-Hermitian metric on the Iwasawa manifold for any $t \in \Delta$.
\end{example}

\begin{rmk}
Note that the same holds true, more in general, on $6$-dimensional nilmanifolds. More precisely, let $X$ be a $6$-dimensional nilmanifold endowed with an Abelian complex structure admitting balanced metrics. Let $\{X_t\}_{t\in(-\varepsilon,\varepsilon)}$ be a differentiable family of deformations of $X_0=X$, where $\epsilon>0$. Then $X_t$ admits a balanced metric for any $t$ close enough to $0$.
In fact, by \cite[Proposition 2.8]{ugarte-villacampa}, non-tori $6$-dimensional nilmanifolds with Abelian complex structures admitting balanced metrics are $\mathfrak{h_3}$ and $\mathfrak{h}_5$. The case $\mathfrak{h}_5$ being treated in Proposition~\ref{Iwasawa-with-abelian-J}, the conclusion follows as before by noticing that, for invariant complex structures on $\mathfrak{h}_3$, the dimension of $H^{2,2}_{BC}(X)$ is always equal to $7$, as computed in \cite[Table 2]{AFR} and \cite[Appendix 6]{LUV}.
\end{rmk}

\begin{rmk}

Let $X$ be the Iwasawa manifold endowed with an Abelian complex structure, and let $X_t$ be any small deformation of $X$. 
Notice that if the complex structure of $X_t$, $t \not=0$, is not Abelian (as it happens in Example~\ref{Iwasawa-with-abelian-J}),
then by \cite[Proposition 3.6]{ugarte-villacampa} $X_t$ never satisfies the weak $\del\delbar$-Lemma, 
so the existence of balanced metric on $X_t$ guaranteed by Proposition~\ref{prop:iwasawa-ab-def} 
cannot be derived from \cite[Theorem~6]{fu-yau}.
\end{rmk}

\subsection{Balanced metrics and \texorpdfstring{$(n-1,n)$}{(n-1,n-1)}-th strong \texorpdfstring{$\del\delbar$}{partialoverlinepartial}-Lemma }

It follows from the numerical characterization in Corollary \ref{numerical-char} that the $(n-1,n)$-th strong $\del\delbar$-Lemma is an open property under deformations. On the other side, it is not closed under deformations.

\begin{prop}\label{openness}\label{no-closed}
For compact complex manifolds, the property of satisfying the $(n-1,n)$-th strong $\del\delbar$-Lemma
\begin{itemize}
 \item is open,
 \item is not closed,
\end{itemize}
under deformations of the complex structure.
\end{prop}

\begin{proof}
We provide here a proof and a counter-example.

\smallskip
\noindent \textbf{Openness.}
Let $\{X_t\}_{t\in B}$ be any differentiable family of compact complex manifolds and suppose that $X_{t_0}$
satisfies the $(n-1,n)$-th strong $\del\delbar$-Lemma for some $t_0\in B$.
By Corollary~\ref{numerical-char} and by the
upper-semi-continuity of the Aeppli numbers (see \cite[Lemme 3.2]{schweitzer}), we have
$$
b_1 \leq 2\, h^{0,1}_{A}(X_{t})\leq 2\, h^{0,1}_{A}(X_{t_0}) = b_1
$$
for all $t\in B$ close enough to $t_0$.
Therefore,
$b_1 = 2\, h^{0,1}_{A}(X_{t})$ and
$X_{t}$ satisfies the $(n-1,n)$-th strong $\del\delbar$-Lemma for all $t$ sufficiently close to $t_0$.

\smallskip
\noindent \textbf{Non-closedness.}
Consider the deformations $\{X_t\}_{t \in B}$ in case {\itshape (1)} given in \cite[\S4]{angella-kasuya-2}.
Here $B$ is an open ball around $0$ in $\mathbb{C}$ and the central fibre $X_0$ is the (holomorphically parallelizable) Nakamura manifold.
The fibres $X_t$ satisfy the $\del\delbar$-Lemma for $t \not=0$, however $\Delta^1(X_0)= 8\not=0$.
Moreover, the central fibre is not sGG \cite{popovici-ugarte}.
\end{proof}

\medskip

We prove now that the $(n-1,n)$-th strong $\del\delbar$-Lemma yields stability of the balanced condition. 
Note that the following result is stated for locally conformally balanced manifolds and that, 
in general, the existence of locally conformally balanced metrics is much weaker than the existence of balanced metrics: 
for instance, there exist many locally conformally balanced nilmanifolds 
not admitting any balanced metric \cite{medori-tomassini-ugarte}.

\begin{thm}\label{conseq}
Let $X$ be a compact complex manifold of complex dimension $n$ with a locally conformally balanced metric, and let $\{X_{t}\}_t$ be a holomorphic family of deformations of $X=X_0$.
If $X$ satisfies the $(n-1,n)$-th strong $\del\delbar$-Lemma, then $X_{t}$ admits a balanced metric for any $t$ sufficiently close to $0$.
\end{thm}

\begin{proof}
In \cite[Theorem 2.5]{angella-ugarte-1} it is proved that, on a compact complex manifold $X$ of complex dimension $n$, if the natural map $\iota^{n-1,\,n}_{BC,\,\delbar} \colon H^{n-1,n}_{BC}(X) \to H^{n-1,n}_{\delbar}(X)$ induced by the identity is injective, then any locally conformally balanced metric is also globally conformally balanced.
Since $X$ satisfies the $(n-1,n)$-th strong $\del\delbar$-Lemma,
by the commutative diagram \eqref{diagg} the map $\iota^{n-1,\,n}_{BC,\,\delbar}$ is injective 
and we can apply \cite[Theorem 2.5]{angella-ugarte-1}
to ensure the existence of a balanced metric on $X$.

Since the $(n-1,n)$-th strong $\del\delbar$-Lemma property is open by Corollary~\ref{openness}, we have that $X_{t}$ also satisfies the $(n-1,n)$-th strong $\del\delbar$-Lemma for any $t$ sufficiently close to $0$, in particular $X_{t}$ satisfies the $(n-1,n)$-th weak $\del\delbar$-Lemma.
Now, Theorem~\ref{fu-yau} implies the existence of a balanced metric on $X_t$ for any $t$ sufficiently close to $0$.

For the sake of completeness, we recall here the main ideas in the argument by J. Fu and S.-T. Yau. 
By the Ehresmann theorem, we look at $X_t$ as $(X,J_t)$ where $\{J_t\}_t$ is a differentiable family 
of complex structures on the fixed smooth manifold $X$.
Let $\omega_0$ be a balanced metric on $(X,J_0)$.
Denote by $\pi_t^{(p,q)} \colon \wedge^{p+q}X\otimes\C \to \wedge^{p,q}_{J_t}X$ the projection onto the space of $(p,q)$-forms with respect to $J_t$.
Note that, since $\omega_0^{n-1}$ is closed, the form $\pi^{(n-1,n-1)}_t\omega_0^{n-1}$ is a $(n-1,n-1)$-form with respect to $J_t$ such that $\delbar_t\pi^{(n-1,n-1)}_t\omega_0^{n-1}$ is $\del_t$-exact. Then, by the $(n-1,n)$-th weak $\del\delbar$-Lemma, there exists a $(n-2,n-1)$-form $\psi_t$ with respect to $J_t$ such that $\delbar_t\pi^{(n-1,n-1)}_t\omega_0^{n-1}=\im\del_t\delbar_t\psi_t$.
Set
$$ \Omega_t \;:=\; \pi^{(n-1,n-1)}_t\omega_0^{n-1} + \im \del_t\psi_t - \im \delbar_t\bar\psi_t \;. $$
It is a real closed $(n-1,n-1)$-form with respect to $J_t$, and we claim that it is also positive. Indeed, this follows by noting that we can choose $\psi_t$ such that $i\psi_t=(\del_t\delbar_t)^*\gamma_t$ where $\gamma_t$ is a solution of the elliptic problem $\square_{BC, t}\gamma_t=\delbar_t\pi^{(n-1,n-1)}_t\omega_0^{n-1}$, the self-adjoint $4$th order elliptic operator $\square_{BC,t}$ being the Bott-Chern Laplacian \cite{schweitzer} with respect to the metric $\sqrt[n-1]{\pi^{(n-1,n-1)}_t\omega_0^{n-1}}$. Then, by elliptic estimates, it follows that $\Omega_t$ is close to $\omega_0^{n-1}$ for $t$ small enough, and hence positive. Therefore, by the Michelsohn trick,
$$ \omega_t \;:=\; \sqrt[n-1]{\Omega_t} $$
gives a balanced Hermitian structure on $(X,J_t)$.
\end{proof}

\begin{example}\label{example1}
As an example of application of Theorem \ref{conseq}, we consider the completely-solvable Nakamura manifold $X$
with the lattice in case {\itshape (ii)} in \cite[Example 2.17]{angella-kasuya-1}. 
More precisely, let $G:=\C\ltimes _{\phi}\C^{2}$, where
$$ \phi\left(x+\im\,y\right) \;:=\;
\left(
\begin{array}{cc}
\exp(x)& 0  \\
0&    \exp(-x)
\end{array}
\right) \in\mathrm{GL}\left(\C^2\right)\;.
$$
For some $a\in\mathbb{R}$, for any $b\in\mathbb{R}$, for any $\Gamma^{\prime\prime} $ lattice of $\C^{2}$, the subgroup $\Gamma := \left(a\,\mathbb{Z}+b\,\im\,\mathbb{Z}\right)\ltimes_\phi \Gamma^{\prime\prime}$ is a lattice of $G$. In particular, we consider the case $b=(2m+1)\pi$ for some integer $m\in\mathbb{Z}$, which is denoted as case {\itshape (ii)} in \cite[Example 2.17]{angella-kasuya-1}.
Consider the completely-solvable solvmanifold $X:=\left.\Gamma\middle\backslash G\right.$.

The compact complex 3-dimensional manifold $X$ satisfies $b_1 =2$ and $h^{0,1}_{BC}(X) = h^{0,1}_{\delbar}(X) = h^{0,1}_{A}(X) = 1$,
so Corollary~\ref{numerical-char} implies that $X$ satisfies the $(2,3)$-th strong $\del\delbar$-Lemma.
(Notwithstanding, $\Delta^2(X)=4 \not=0$ so $X$ does not satisfy the $\del\delbar$-Lemma \cite{angella-tomassini-3}. Moreover, small deformations $X_t$ of $X$ have $\dim_\C H^{2,2}_{BC}(X_t)\in\{3,7,11\}$, see \cite[Table 6]{angella-kasuya-1}, whence we cannot apply Proposition \ref{prop:stability-under-const-dim-bc} in general.)

Moreover, the compact complex manifold $X$ has a balanced metric. It is given as follows.
Consider holomorphic coordinates $\left\{z_1,\, z_2,\, z_3\right\}$ on $G$, where $\left\{ z_1 \right\}$ is the holomorphic coordinate on the factor $\C$ and $\left\{ z_2, z_3 \right\}$ are holomorphic coordinates on the factor $\C^2$. The form
$$
\omega= \frac{\im}{2} ( dz_1 \wedge d\bar{z}_1
+ e^{-z_1 -\bar{z}_1} dz_2 \wedge d\bar{z}_2 + e^{z_1 +\bar{z}_1} dz_3 \wedge d\bar{z}_3 )
$$
is well defined on the quotient $X$, and it yields a balanced metric on $X$.
Now, by Theorem~\ref{conseq} any small deformation of the non-$\partial \bar\partial$-manifold $X$ also admits a balanced metric.
\end{example}

\end{document}